\documentclass[reqno]{amsproc}

\usepackage{hyperref}

\hypersetup{
colorlinks=true,
linkcolor=blue,
anchorcolor=blue,
citecolor=blue
}
\usepackage[T1]{fontenc}
\usepackage{lmodern}

\usepackage{lmodern}
\usepackage{amsmath,amssymb,amsthm,mathtools}
\usepackage[letterpaper,margin=1.03in]{geometry}
\usepackage{microtype}
\usepackage{array,booktabs,mathrsfs}
\usepackage{graphicx}
\usepackage{color}

\newtheorem{theorem}{Theorem}[section]
\newtheorem{lemma}[theorem]{Lemma}
\newtheorem{proposition}[theorem]{Proposition}
\newtheorem{corollary}[theorem]{Corollary}

\newtheorem{remark}[theorem]{Remark}

\numberwithin{equation}{section}

\def\Ric{\text{Ric}}

\def\Ric{\operatorname{Ric}}
\def\Rm{\operatorname{Rm}}

\newcommand{\R}{\mathbb R}
\newcommand{\g}{\mathfrak g}
\newcommand{\n}{\mathfrak n}
\newcommand{\cM}{\mathcal M}
\newcommand{\cL}{\mathcal L}
\newcommand{\cJ}{\mathcal J}
\newcommand{\cR}{\mathcal R}

\newcommand{\ip}[2]{\langle #1,#2\rangle}
\newcommand{\norm}[1]{\lVert #1\rVert}
\DeclareMathOperator{\Lie}{Lie}

\DeclareMathOperator{\tr}{tr}

\DeclareMathOperator{\Span}{span}
\DeclareMathOperator{\ad}{ad}
\DeclareMathOperator{\Ad}{Ad}

\DeclareMathOperator{\Ann}{Ann}

\newcommand{\SO}{{\mathsf{SO}}}

\def \R {{\mathbb R}}

\def\Ad {{\rm Ad}}

\def\g {{\mathfrak{g}}}

\def\Rm{{\operatorname R}}
\def\Ric{{\operatorname{Ric}}}

\def\ad{{\operatorname{ad}}}

\numberwithin{equation}{section}

\begin{document}

	\title[Simons' theorem]{A single induction proof of  Simons'  Riemannian 
  holonomy theorem}

\author{Lei Ni}\thanks{We would like to thank SICM where the first version of
 a single induction proof was worked out. }
\address{Lei Ni. Zhejiang Normal University, Jinhua, Zhejiang, China, 321004; 
 Department of Mathematics, University of California, San Diego, La Jolla, CA 92093, USA}
\email{lni.math.ucsd@gmail.com; leni@zjnu.edu.cn}

\subjclass[2010]{}

\begin{abstract} We give a detailed and complete  algebraic proof of Simons' 
  Riemannian holonomy theorem by a single 
strong induction on the dimension of the linear span of the curvature orbit.
The proof uses the construction of flats and root centralizers, but replaces
the detailed analysis of the common zero-weight space, the most technical core of previous algebraic 
 proofs,  by a trace-vanishing
lemma.  The lemma applies to curvature operators contained in an ideal that
annihilates the curvature module.  Restriction to a common geodesic subspace produces a proper
invariant kernel.  Induction makes this kernel fixed by the holonomy action, after which
the trace lemma shows that the root centralizers span the ambient space.
Their intersection recovers the maximal flat, and irreducibility completes
the proof.  For the completeness we also include the detailed derivation of 
the local symmetry of the underlying Riemanian manifold
 from the algebraic theorem.
\end{abstract}

\maketitle

\begin{flushright} {{\it `And I cherish more than anything else\\  the Analogies, my most trustworthy masters.\\
 They know all the secrets of Nature.'}--Kepler}\\

 {{\it `Invert, always invert'}}-- Munger

\end{flushright}
\section{Introduction}

Berger's holonomy theorem and Simons' algebraic proof are fundamental results
in Riemannian geometry \cite{Berger,Simons}.  We begin with the geometric
statement.

\begin{theorem}[Berger--Simons]\label{thm:geometric}
Let $(M^n,g_M)$ be a connected Riemannian manifold, $n\geq2$, whose restricted
holonomy group acts irreducibly on a tangent space.  If this group is not
transitive on the unit sphere, then $M$ is locally symmetric and has rank at
least two.
\end{theorem}

A \emph{Riemannian holonomy system} is a triple $S=(V,\Rm,G)$ consisting of a
Euclidean space $V$ of dimension $n\geq2$, a connected compact subgroup
$G\subset\SO(V)$, and an algebraic curvature tensor $\Rm$ such that
$\Rm_{x,y}\in\g:=\Lie(G)$ for all $x,y\in V$.  The system is
\emph{irreducible} if $G$ acts irreducibly on $V$, and \emph{symmetric} if
$g(\Rm)=\Rm$ for every $g\in G$.  The action on curvature tensors and our sign
conventions are specified in Section~\ref{sec:prelim} (cf. Section 2 of \cite{N}).

\begin{theorem}[Simons]\label{thm:usual}
An irreducible Riemannian holonomy system $(V,\Rm,G)$ is symmetric if $G$ is
not transitive on the unit sphere of $V$.
\end{theorem}

For a holonomy system, write
\begin{equation}\label{eq:module}
 \cM:=G(\Rm)=\Span_{\R}\{g(\Rm):g\in G\},\qquad
 \g^\Rm:=\Span_{\R}\{Q_{x,y}:Q\in\cM,\ x,y\in V\}.
\end{equation}
Thus $G(\Rm)$ is a \emph{linear space} spanned by the orbit.
The subspace $\g^\Rm$ is an ideal of $\g$.  Its orthogonal complementary ideal
is denoted by $\g_\Rm$.  The connected analytic subgroup with Lie algebra
$\g^{\Rm}$ is denoted by $G^\Rm$; in the irreducible setting it is closed and
compact.  These facts shall be  proved in Section \ref{sec:prelim}.

\begin{theorem}[Simons]\label{thm:main}
Let $S=(V,\Rm,G)$ be an irreducible Riemannian holonomy system.  If $G^\Rm$ is
not transitive on the unit sphere, then $S$ is symmetric and has rank at
least two.
\end{theorem}

Here a flat is a subspace $W\subset V$ such that $Q_{u,v}=0$ for all
$u,v\in W$ and $Q\in G(\Rm)$; the rank is the largest dimension of such a
subspace.  In particular, when $\Rm=0$, the rank is $\dim V$.
Theorem~\ref{thm:usual} follows from Theorem~\ref{thm:main} because
$G^\Rm\subset G$.

The purpose of this paper is to prove Theorem~\ref{thm:main} by a single
strong induction on
\begin{equation}\label{eq:index}
 \kappa(S):=\dim G(\Rm).
\end{equation}
No induction on $\dim V$ or $\dim\g^\Rm$ is used.  The proof of Simons 
\cite{Sim} is a nested induction on $(\dim (V), \dim(\mathfrak{g}^\Rm))$. In \cite{N}, 
the author  produced an alternate argument using the induction on  $(\dim (V), \dim(G(\Rm)))$.
 The argument here uses Simons' flats and
root centralizers, together with the analytic proof of the flat lemma
suggested by N.~Wallach from \cite{N}.  The main new ingredient is the
trace-vanishing Lemma~\ref{lem:trace}.  

The analogy with  Lie theory is a guide to the constructions:
\begin{center}
\renewcommand{\arraystretch}{1.18}
\begin{tabular}{@{}p{0.43\textwidth}p{0.48\textwidth}@{}}
\toprule
{\color{blue}Lie theory} & {\color{blue}Holonomy system} \\
\midrule
{\color{blue} Cartan subalgebra $\mathfrak h$} & {\color{blue}Maximal flat $W$ }\\
{\color{blue} Root $\alpha$} & {\color{blue}Linear form $\ell_{k,Q}$ }\\
{\color{blue}Root hyperplane $\ker\alpha$} & {\color{blue}Hyperplane $U_{k,Q}=\ker\ell_{k,Q}$ }\\
{\color{blue} Centralizer of $\ker\alpha$} & {\color{blue}$E_{k,Q}=C(U_{k,Q})$} \\
{\color{blue}Nonzero root spaces}& {\color{blue}Nonzero common Jacobi eigendirections }\\
\bottomrule
\end{tabular}
\end{center}
These are analogies, not identifications with a root system before symmetry
has been proved.  In particular, the auxiliary space $\g\oplus V$ in
Lemma~\ref{lem:trace} is not assumed to be a Lie algebra.

Section~\ref{sec:geometric} retains the geometric consequence and the detailed 
Kostant argument from the algebraic approach in \cite{N}.  Sections~\ref{sec:flats}--\ref{sec:induction}
prove Theorem~\ref{thm:main}.    A different, geometric
proof using Euclidean submanifold and normal holonomy theory can be found in 
Olmos \cite{Olmos}. 

\section{Algebraic preliminaries}\label{sec:prelim}

\subsection{Algebraic curvature operators and the group action}
An algebraic curvature tensor is a linear map
$Q:\Lambda^2V\longrightarrow\mathfrak{so}(V)$ satisfying
\begin{align}
 Q(x,y,z,t)&:=\ip{Q_{x,y}z}{t}
             =Q(z,t,x,y),\label{eq:cur12-34}\\
 \mathfrak{S}(Q_{x,y}z)&=0, {\mathfrak S} \text { is the cyclic permutations  on the positions of }
 x, y, z.\label{eq:bianchi}
\end{align}
We denote by  $\cR(V)=S_B^2(\Lambda^2V)$ for the space of these tensors.
As in Section 2 of \cite{N}, the identification with symmetric operators on $\Lambda^2V$ uses 
\[
 (x\wedge y)z=\ip{y}{z}x-\ip{x}{z}y,\qquad
 \ip{A}{B}_{\mathfrak{so}(V)}=-\tfrac12\tr_V(AB).
\]
In particular,
\begin{equation}\label{eq:wedge-pairing}
 \ip{A}{x\wedge y}_{\mathfrak{so}(V)}=-\ip{Ax}{y},
 \qquad Q(x\wedge y)=-Q_{x,y}.
\end{equation}
Together it interprets $Q$ as an element in $S_B^2(\Lambda^2V)$. We use
the same letter for the tensor and the operator.

For $g\in G$,
\begin{equation}\label{eq:group-action}
 (g(Q))_{x,y}=gQ_{g^{-1}x,g^{-1}y}g^{-1}.
\end{equation}
The induced inner product on $\cR(V)$ is $G$-invariant.  If
$B\in\g$, differentiation gives
\begin{align}
 B(Q)_{x,y}&=[B,Q_{x,y}]-Q_{Bx,y}-Q_{x,By},\label{eq:infinitesimal}\\
 B(Q)(x,y,z,t)&=-Q(Bx,y,z,t)-Q(x,By,z,t)\notag\\
 &\hspace{13mm}-Q(x,y,Bz,t)-Q(x,y,z,Bt).\label{eq:four-action}
\end{align}
The endomorphism $Q\mapsto B(Q)$ of any invariant curvature subspace is
skew-adjoint.  Thus
\begin{equation}\label{eq:skew-norm}
 \norm{B(Q)}^2=-\ip{B^2(Q)}{Q}.
\end{equation}
Here $B^2(Q)$ means $B(B(Q))$.

\begin{lemma}[Sectional curvature determines the curvature operator]\label{lem:sectional}
An algebraic curvature tensor is determined by the values $Q(x,y,x,y)$.
In particular, if these values vanish for all $x,y$, then $Q=0$.
\end{lemma}
\begin{proof}
Polarization in $y$ gives $Q(x,y,x,z)=0$.  Polarizing in $x$ then gives
$Q(x,y,u,z)+Q(u,y,x,z)=0$. Together with (\ref{eq:cur12-34}) it implies that the three terms in the first
Bianchi identity, paired with $z$, are equal.  Their sum is zero, so each
term is zero. See \cite{Sakai} for a proof using calculus.
\end{proof}

\subsection{The curvature-generated ideal} The first elementary result is from \cite{Sim}. 
\begin{lemma}\label{lem:ideal}
For a holonomy system $(V,\Rm,G)$, the space $\g^\Rm$ in
\eqref{eq:module} is an ideal of $\g$.  If
\[
 \g=\g^\Rm\oplus\g_\Rm
\]
is its orthogonal decomposition with respect to the invariant inner product
on $\mathfrak{so}(V)$, then
\begin{equation}\label{eq:complement-annihilates}
 B(Q)=0\qquad\forall B\in\g_\Rm,\ Q\in\cM.
\end{equation}
Moreover, if $T\in\cM$, then
\begin{equation}\label{eq:submodule-ideal}
 G(T)\subset\cM,\qquad \g^T\subset\g^\Rm.
\end{equation}
\end{lemma}
\begin{proof}
For $B\in\g$ and $Q\in\cM$, equation~\eqref{eq:infinitesimal} gives
\[
 [B,Q_{x,y}]=B(Q)_{x,y}+Q_{Bx,y}+Q_{x,By}\in\g^\Rm,
\]
because $B(Q)\in\cM$.  Hence $\g^\Rm$ is an ideal.  Its orthogonal
complement (in $\g$, a compact Lie algebra) is also an ideal, and $[\g_\Rm,\g^\Rm]=0$.

For $B\in\g_\Rm$, regard $Q$ as a symmetric operator on $\Lambda^2V$.
Its image lies in $\g^\Rm$, so $\ad_B\cdot Q\in \g_Rm\cap \g^\Rm$ thus $=0$.  
Since $Q$ is symmetric
and $\ad_B$ is skew-adjoint, we have $0=(\ad_B\cdot Q)^{tr}=-Q \cdot \ad_B$.  The operator form of
\eqref{eq:infinitesimal} is
$B(Q)=\ad_B\cdot Q-Q\cdot\ad_B$, proving
\eqref{eq:complement-annihilates}.
Finally, $\cM$ is $G$-invariant and contains $T$, which proves both
inclusions in~\eqref{eq:submodule-ideal}.
\end{proof}
The result below is standard. We include its proof for the completeness.
\begin{lemma}[Closedness of subgroups whose Lie algebras are ideals]\label{lem:closed}
Let $G\subset\SO(V)$ be connected and compact and act irreducibly on $V$.
Every ideal of $\g$ integrates to a closed connected subgroup of $G$.
In particular, $G^\Rm$ is compact.
\end{lemma}
\begin{proof}
First $\dim\mathfrak z(\g)\leq1$.  Indeed, if
$0\neq A\in\mathfrak z(\g)$, the symmetric endomorphism $A^2$ commutes
with $G$.  Its eigenspaces are invariant, so irreducibility gives
$A^2=-a^2\operatorname{id}$ with $a>0$.  If
$B\in\mathfrak z(\g)$, then $AB$ is likewise symmetric and commutes
with $G$, hence is scalar.  Thus $B$ is a scalar multiple of $A$.

Use the elementary decomposition of a compact Lie algebra into its center
and compact simple ideals:
\[
 \g=\mathfrak z(\g)\oplus\g_1\oplus\cdots\oplus\g_s.
\]
An ideal is a sum of some $\g_j$ and a subspace of the center.  Since the
center has dimension at most one, this last subspace is either zero or the
whole center.  The simple factors integrate to compact connected subgroups,
and the connected central subgroup is a torus of dimension at most one.
Their product is compact and with the prescribed Lie algebra.
Uniqueness of the connected analytic subgroup proves the assertion. 
\end{proof}
 This proof can also be found in 
\cite[Appendix~5]{KN}, so is the next result.
 We include a simple proof for completeness. 
\begin{corollary}\label{cor:irreducible-closed}
A connected Lie subgroup $H\subset\SO(V)$ acting irreducibly on $V$ is
closed and compact.
\end{corollary}
\begin{proof}
Its closure $\overline H$ is compact, connected and irreducible.
Since $\Ad_h\Lie(H)=\Lie(H)$ for $h\in H$, continuity gives the same
invariance under $\overline H$.  Thus $\Lie(H)$ is an ideal of
$\Lie(\overline H)$.  Lemma~\ref{lem:closed} shows that the corresponding
analytic subgroup, namely $H$, is closed in $\overline H$.
\end{proof}

\section{From algebraic symmetry to parallel curvature}\label{sec:geometric}

This section is from Section 3 of \cite{N} explaining the geometric consequence of
Theorem~\ref{thm:usual}.  It is logically separated from the induction in
Sections~\ref{sec:flats}--\ref{sec:induction}. We include it for the completeness.

\subsection{Kostant's argument for a symmetric system}
For a compact algebra $\g\subset\mathfrak{so}(V)$, let
\[
 B(A_1,A_2)=K_{\g}(A_1,A_2)-2\ip{A_1}{A_2}_{\mathfrak{so}(V)}.
\]
Here $K(\cdot, \cdot)$ is Killing form of $\g$, which is semi-negative definite due 
to that $\mathfrak{g}$ is compact.  Let $\mathcal T:\g\to\g$ be defined by
$B(A_1,A_2)=\ip{\mathcal T A_1}{A_2}_{\mathfrak{so}(V)}$, and let
$\pi_{\g}:\Lambda^2V\to\g$ be orthogonal projection. $\mathcal T$ is
 negative definite since $B<0$.

\begin{theorem}[Kostant; cf.\ \cite{Simons,Sakai, N}]\label{thm:kostant}
For an irreducible symmetric holonomy system $(V,Q,G)$ there exists
$\lambda\in\R$ such that
\begin{equation}\label{eq:kostant}
 Q_{x,y}=-\lambda\mathcal T^{-1}\pi_{\g}(x\wedge y).
\end{equation}
If $Q\neq0$, then $\lambda\neq0$, and
\begin{align}
 Q(x,y,z,t)&=\lambda^{-1}B(Q_{x,y},Q_{z,t}),\label{eq:kostant-pair}\\
 \Ric_Q&=-\tfrac{\lambda}{2}\ip{\cdot}{\cdot},\label{eq:kostant-ricci}\\
 Q(x,y,x,y)=0&\quad\Longleftrightarrow\quad Q_{x,y}=0.\label{eq:kostant-zero}
\end{align}
Here $\Ric_Q(x,y)=\sum_j Q(e_j,x,e_j,y)$.
\end{theorem}
\begin{proof}
For $Q=0$ take $\lambda=0$ in~\eqref{eq:kostant}.  

Suppose $Q\neq0$.
On $\cJ=\g\oplus V$ recall the Cartan construction 
\[
 [A_1,A_2]_{\cJ}=[A_1,A_2],\qquad [A,x]_{\cJ}=A(x),\qquad [x,y]_{\cJ}=Q_{x,y}.
\]
The Jacobi identity follows from the first Bianchi identity and
$A(Q)=0$ for every $A\in\g$.  Thus  $\cJ$  is
a Lie algebra.  Denote its Killing form by $B_{\cJ}$.
A direct calculation (cf. Section 3 of \cite{N} for details) gives
\[
 B_{\cJ}|_{\g}(\cdot, \cdot)=B(\cdot, \cdot),\qquad B_{\cJ}(\g,V)=0.
\]
The restriction to $V$ is a $G$-invariant symmetric bilinear form, so
irreducibility gives $B_{\cJ}|_{V\times V}=\lambda\ip{\cdot}{\cdot}$. In the
 below we simply use $B$ for $B_{\cJ}$. If $\lambda=0$, invariance yields
\[
 B(Q_{x,y},Q_{x,y})
 =B_{\cJ}([x,y],[x,y])
 =B_{\cJ}(x,[y,[x,y]])=0.
\]
Negative definiteness of $B$ on $\g$ imply $Q=0$, a contradiction.

Similarly 
\[
 \lambda\ip{Q_{x,y}z}{t}
 =B_{\cJ}([[x,y],z],t)
 =B(Q_{x,y},Q_{z,t}),
\]
which proves~\eqref{eq:kostant-pair} and~\eqref{eq:kostant-zero}.
Let $\mathfrak k=\Span_{\mathbb{R}}\{Q_{x,y}\}$.  By Lemma \ref{lem:ideal} it is an ideal of $\g$.
If $\mathfrak{k}^\perp$ is $B$-orthogonal to $\mathfrak k$, 
then $\forall y \in V, A\in \mathfrak{k}^\perp$, 
\[
 \lambda\norm{A(y)}^2
 =B_{\cJ}([A,y],[A,y])
 =B_{\cJ}(A,[y,[A,y]])
 =B(A,Q_{y,Ay})=0.
\]
Thus $\mathfrak{k}^\perp=0$, so $\mathfrak k=\g$.  Combining
\eqref{eq:kostant-pair} with~\eqref{eq:wedge-pairing} gives
\[
 \ip{Q_{x,y}}{\mathcal T Q_{z,t}}_{\mathfrak{so}(V)}=\lambda Q(x,y,z,t)=\lambda \ip{Q_{x, y}z}{t}
 =-\lambda\ip{Q_{x,y}}{\pi_{\g}(z\wedge t)}_{\mathfrak{so}(V)}.
\]
Since the $Q_{x,y}$ span $\g$, formula~\eqref{eq:kostant} follows.

Finally, $\ad_x$  on $\g\oplus V$ is off diagonal, decomposing into maps $\alpha_x:
A\mapsto-A(x)$ and $\beta_x :z\mapsto Q_{x,z}$.  
Direct calculation gives $\tr_{\mathfrak{g}}(\beta_x \alpha_x)
 =\tr_{V}(\alpha_x,\beta_x)$. Hence
\[
 B_{\cJ}(x,x)=\tr_{\cJ}(\ad_x^2)=\tr_{\mathfrak{g}}(\beta_x \alpha_x)
 +\tr_{V}(\alpha_x,\beta_x)=2\tr_{V}(\alpha_x,\beta_x)
 =-2\sum_j\ip{Q_{x,e_j}x}{e_j}=-2\Ric_Q(x,x).
\]
This proves~\eqref{eq:kostant-ricci}.
\end{proof}

\subsection{Curvature and holonomy} This part is
 Lemma 2.2  and the second part of Section 3 of \cite{N}.
For the Levi--Civita connection $D$ we use the convention
\begin{equation}\label{eq:geometric-sign}
 \Rm_{X,Y}=-D_XD_Y+D_YD_X+D_{[X,Y]}.
\end{equation}
Let $H_p^0$ be the restricted holonomy group and $\mathfrak h_p$ its Lie
algebra.  For a path $\gamma$ from $q$ to $p$, write $\tau_\gamma$ for
parallel transport and $\tau_\gamma(R_q)$ for the transported tensor.

\begin{lemma}\label{lem:curvature-inclusion}
For every such path and every $x,y\in T_pM$,
\[
 \bigl(\tau_\gamma(\Rm_q)\bigr)_{x,y}\in\mathfrak h_p.
\]
\end{lemma}
\begin{proof}
Recall the elementary infinitesimal part of the holonomy theorem
\cite{AmbroseSinger,KN}.  At $p$, choose commuting coordinate vector fields
$X,Y$ with $X_p=x$ and $Y_p=y$.  Let $\Psi(s,t)$ be parallel transport
around the coordinate rectangle with side parameters $s,t$, with orientation
chosen to agree with~\eqref{eq:geometric-sign}.  Then
\[
 \Psi(s,0)=\Psi(0,t)=\operatorname{id},\qquad
 \left.\frac{\partial^2\Psi}{\partial s\partial t}\right|_{(0,0)}=\Rm_{x,y}.
\]
The second identity follows by extending an initial vector first parallel
along the $X$-curves and then along the $Y$-curves, and differentiating the
parallel-transport equations; the resulting mixed derivative is
$D_YD_X-D_XD_Y$ at $p$.
For each fixed $s$, $\partial_t\Psi(s,0)\in\mathfrak h_p$.
Differentiating in $s$ proves $\Rm_{x,y}\in\mathfrak h_p$.

Parallel transport conjugates restricted holonomy groups:
$\tau_\gamma H_q^0\tau_\gamma^{-1}=H_p^0$.
Therefore it conjugates their Lie algebras and gives the asserted inclusion
at $q$.  Only this inclusion, not the reverse spanning assertion of the full
Ambrose--Singer theorem, is used here.
\end{proof}

\begin{proof}[Derivation of Theorem~\ref{thm:geometric} from Theorem~\ref{thm:usual}]
The restricted holonomy groups are connected Lie subgroups of the orthogonal
groups.  Under the irreducibility hypothesis, they are compact by
Corollary~\ref{cor:irreducible-closed}.
For $n=2$, a connected irreducible subgroup of $\SO(2)$ is $\SO(2)$ itself,
so the nontransitive case does not occur.  Assume $n\geq3$ and
nontransitivity.

Lemma~\ref{lem:curvature-inclusion} and Theorem~\ref{thm:usual} show that
$(T_pM,\Rm_p,H_p^0)$ is symmetric.  Consequently its Ricci tensor is invariant
under the irreducible group $H_p^0$, and hence
$\Ric=f\,g_M$.  The contracted second Bianchi identity makes $f=\beta$
constant on $M$.

Fix $p$.  For every path $\gamma:q\to p$, the transported tensor
$Q=\tau_\gamma(\Rm_q)$ defines an irreducible holonomy system on $T_pM$ with
structure group $H_p^0$.  It is symmetric by Theorem~\ref{thm:usual}.
The operators $\mathcal T$ and $\pi_{\mathfrak h_p}$ in
Theorem~\ref{thm:kostant} depend only on $\mathfrak h_p\subset\mathfrak{so}(T_pM)$,
not on $q$ or $\gamma$.  Moreover,
$\Ric_Q=\beta\,g_p$, so~\eqref{eq:kostant-ricci} gives
$\lambda_Q=-2\beta$.  This remains valid when $Q=0$, for which
$\lambda_Q=0$.  Formula~\eqref{eq:kostant} therefore makes
$\tau_\gamma(R_q)$ independent of $q$ and $\gamma$, and equal to $R_p$.
Thus $D\Rm=0$.

The flat lemma below provides a flat of dimension at least two.  For a
locally symmetric metric, the maximal dimension of tangent subspaces with
$\Rm_{W,W}=0$ is the usual geometric rank; equivalently, these subspaces
exponentiate locally to totally geodesic flat submanifolds. 
\end{proof}

\section{Flats, Jacobi operators, and root centralizers}\label{sec:flats}

Throughout this section $S=(V,\Rm,G)$ ($\dim(V)\ge 2$) is irreducible and $G^\Rm$ is
nontransitive.  Set $\cM=G(\Rm)$.
A subspace $E\subset V$ is called \emph{geodesic} if
\begin{equation}\label{eq:geodesic}
 Q_{x,y}z\in E,\qquad \forall x,y,z\in E,\ Q\in\cM.
\end{equation}
The word always means geodesic for \emph{every} tensor in $\cM$.
All the constructions in this section are made without an induction
hypothesis.

\subsection{The flat lemma}
The pair symmetry of the curvature operator gives the useful equivalence
\begin{equation}\label{eq:flat-criterion}
 W\text{ is flat}
 \quad\Longleftrightarrow\quad
 \g^\Rm(W)\subset W^\perp.
\end{equation}
In particular $G$ sends flats to flats.

\begin{proposition}[Flat lemma]\label{prop:flat}
Every vector of $V$ belongs to a flat of dimension at least two.
\end{proposition}
\begin{proof}
If $\Rm=0$, take $W=V$.  Otherwise $H:=G^\Rm$ is a nontrivial compact normal
subgroup of $G$.  The fixed space $V^H$ is $G$-invariant.  Irreducibility
and faithfulness imply $V^H=0$.

Fix a unit vector $u$.  By nontransitivity choose a unit vector $v$ outside
its $H$-orbit.  The function
\[
 f(h)=\ip{u}{hv},\qquad h\in H,
\]
attains its maximum at $h_0$.  Differentiating gives
$\ip{u}{Bh_0v}=0$ for every $B\in\g^\Rm$.  Hence, for every $Q\in\cM$,
\[
 0=\ip{u}{Q_{x,y}h_0v}=\ip{Q_{h_0v,u}x}{y}
 \qquad\forall x,y\in V,
\]
so $Q_{u,h_0v}=0$.
The vectors $u$ and $h_0v$ are independent: $h_0v\neq u$ by the choice of
$v$, and $h_0v=-u$ would make the maximum equal to $-1$, forcing every
$hv=-u$, contrary to $V^H=0$.  Thus
$\Span_{\mathbb{R}}\{u,h_0v\}$ is the required flat.  The zero vector lies in any such
flat.
\end{proof}

\subsection{A universal zero-pair identity}
If $u,v\in V$ satisfy $P_{u,v}=0$ for every $P\in\cM$, then
\eqref{eq:infinitesimal}, applied to $B(P)\in\cM$, gives
\begin{equation}\label{eq:zero-pair}
 P_{Bu,v}+P_{u,Bv}=0,
 \qquad \forall B\in\g,\ P\in\cM.
\end{equation}
This identity applies both to pairs in a flat and to the centralizer pairs
introduced below.

\subsection{Commuting Jacobi operators}
Fix a flat $W$.  For $u,v\in W$ and $Q\in\cM$, define
\[
 T_Q^{u,v}z:=Q_{u,z}v.
\]
The first Bianchi identity and $Q_{u,v}=0$ imply
\begin{equation}\label{eq:jacobi-symmetric}
 T_Q^{u,v}=T_Q^{v,u},\qquad (T_Q^{u,v})^{\tr}=T_Q^{u,v}.
\end{equation}
Here $A^{tr}$ denotes the traspose of $A$. 
There is also a mixed product identity:
\begin{equation}\label{eq:jacobi-product}
 T_Q^{u,v}T_P^{s,t}=T_Q^{s,t}T_P^{u,v}
 \qquad\forall u,v,s,t\in W,\ P,Q\in\cM.
\end{equation}
Indeed, by~\eqref{eq:zero-pair}, Bianchi, and
\eqref{eq:jacobi-symmetric},
\begin{align*}
 T_Q^{u,v}T_P^{s,t}z
 &=Q_{u,P_{s,z}t}v
 =-Q_{P_{s,z}u,t}v
 =-Q_{P_{u,z}s,t}v\\
 &=Q_{t,P_{u,z}s}v
 =Q_{v,P_{u,z}s}t
 =-Q_{P_{u,z}v,s}t
 =T_Q^{s,t}T_P^{u,v}z.
\end{align*}
Taking $P=Q$ proves that, for each fixed $Q$, the symmetric operators
$T_Q^{u,v}$ commute.

Choose a common orthonormal eigenbasis $X_1^Q,\ldots,X_n^Q$:
\begin{equation}\label{eq:jacobi-eigen}
 T_Q^{u,v}X_k^Q=\Lambda_k^Q(u,v)X_k^Q.
\end{equation}
Each $\Lambda_k^Q$ is a symmetric bilinear form on $W$.  The operators
annihilate $W$, so a vector with $\Lambda_k^Q\neq0$ is orthogonal to $W$.
The simultaneous diagonalization   is asserted only for a fixed $Q$.

\begin{lemma}[Rank-one lemma]\label{lem:rank-one}
If $\Lambda_k^Q\neq0$, it has rank one.  With
$U_{k,Q}:=\ker\Lambda_k^Q$, one has
\begin{equation}\label{eq:root-annihilation}
 P_{s,X_k^Q}=0, \quad \forall  s\in U_{k,Q},\ P\in\cM.
\end{equation}
\end{lemma}
\begin{proof}
Suppress the superscript on $X_k$.  Choose $a\in W$ with
$\Lambda_k^Q(a,a)\neq0$ and initially set
$U=\{s\in W:\Lambda_k^Q(a,s)=0\}$.  For $s\in U$, apply
\eqref{eq:zero-pair} to $B=Q_{a,X_k}$ and the pair $(s,a)$:
\begin{align*}
 \Lambda_k^Q(a,a)P_{s,X_k}
 &=P_{s,Q_{a,X_k}a}
 =-P_{Q_{a,X_k}s,a}\\
 &=-\Lambda_k^Q(a,s)P_{X_k,a}=0.
\end{align*}
This proves~\eqref{eq:root-annihilation} on $U$.  In particular,
$\Lambda_k^Q(s,t)X_k=Q_{s,X_k}t=0$ for $s\in U$, $t\in W$.
Thus the hyperplane $U$ is the kernel of $\Lambda_k^Q$, which has rank one.
\end{proof}

For each nonzero form choose a unit vector $e_{k,Q}\in U_{k,Q}^\perp\cap W$.
There is a nonzero real number $\lambda_{k,Q}$ such that
\begin{equation}\label{eq:rank-one-form}
 \Lambda_k^Q=\lambda_{k,Q}\ell_{k,Q}\otimes\ell_{k,Q},
 \qquad \ell_{k,Q}(u)=\ip{u}{e_{k,Q}}.
\end{equation}
No sign condition is imposed on $\lambda_{k,Q}$.

\subsection{Root centralizers}
For $U\subset W$ define
\begin{equation}\label{eq:centralizer}
 C(U):=\{z\in V:P_{u,z}=0\text{ for every }u\in U,\ P\in\cM\}.
\end{equation}
We pointed above that (\ref{eq:zero-pair}) holds for $u\in U$ and $v\in C(U)$. 
Linearity gives
\begin{equation}\label{eq:centralizer-intersection}
 C(U_1+U_2)=C(U_1)\cap C(U_2).
\end{equation}
If $W$ is maximal, then
\begin{equation}\label{eq:self-centralizer}
 C(W)=W:
\end{equation}
indeed $z\in C(W)$ makes $W+\R z$ flat.

For a maximal flat $W$ and a nonzero $\Lambda_k^Q$, put
\[
 E_{k,Q}:=C(U_{k,Q}).
\]
By Lemma~\ref{lem:rank-one},
\begin{equation}\label{eq:root-containments}
 U_{k,Q}\subsetneq W\subsetneq E_{k,Q},\qquad X_k^Q\in E_{k,Q}.
\end{equation}

\begin{proposition}[Geodesy and properness]\label{prop:centralizers}
Every $C(U)$ is geodesic.  If $\Rm\neq0$, $\dim W\geq2$, and
$\Lambda_k^Q\neq0$, then $E_{k,Q}\subsetneq V$.  Moreover the restriction
of $\cM$ to $E_{k,Q}$ is nonzero.
\end{proposition}
\begin{proof}
Let $E=C(U)$ and let $a,b,z\in E$.  For $u\in U$, Bianchi gives
$P_{a,b}u=0$, since $P_{u,a}=P_{u,b}=0$.
Taking $B=P_{a,b}$ in~\eqref{eq:zero-pair}, for an arbitrary $Q\in\cM$,
we obtain
\[
 Q_{u,P_{a,b}z}=-Q_{P_{a,b}u,z}=0.
\]
Thus $P_{a,b}z\in C(U)$.

The common nullity space
\[
 N_0:=\{u\in V:P_{u,v}=0\text{ for all }v\in V,\ P\in\cM\}
\]
is $G$-invariant.  If $E_{k,Q}=V$, then $U_{k,Q}\subset N_0$.
Since $\dim W\geq2$, the hyperplane $U_{k,Q}$ is nonzero.
Irreducibility gives $N_0=V$, contradicting
$\Rm\neq0$.  This proves properness.
Finally,
\[
 Q_{e_{k,Q},X_k^Q}e_{k,Q}=\lambda_{k,Q}X_k^Q\neq0,
\]
and all three arguments belong to $E_{k,Q}$.
\end{proof}

\begin{lemma}[Duality]\label{lem:duality}
For a root centralizer $E_{k,Q}$,
\[
 U_{k,Q}=\{u\in W:P_{u,z}=0\text{ for all }z\in E_{k,Q},\ P\in\cM\}.
\]
If $E_{k,Q}\neq E_{l,P}$, then $E_{k,Q}\cap E_{l,P}=W$.
\end{lemma}
\begin{proof}
The left hand side is a subset of the right hand side  by the definition of $C(U_{k,Q})$.
If $u\notin U_{k,Q}$, then
\[
 Q_{u,X_k^Q}e_{k,Q}
 =\lambda_{k,Q}\ell_{k,Q}(u)X_k^Q\neq0,
\]
so the reverse inclusion follows.  Distinct centralizers have distinct
hyperplanes.  The sum of two distinct hyperplanes in $W$ is $W$, and
\eqref{eq:centralizer-intersection}--\eqref{eq:self-centralizer} give the
intersection assertion.
\end{proof}

\section{Restriction kernels and a trace-vanishing lemma}\label{sec:trace}

\subsection{The restriction-kernel theorem}
Let $E\subsetneq V$ be geodesic for $\cM=G(\Rm)$ and define
\begin{align}
 r_E:\cM&\longrightarrow\cR(E),& r_E(Q)&=Q|_{E^4},\label{eq:restriction}\\
 K_E&:=\ker r_E,&
 J_E&:=\{B\in\g:B(\cM)\subset K_E\}.\label{eq:kernel-ideal}
\end{align}
Because $E$ is geodesic, $r_E(Q)=0$ is equivalent to
$Q_{x,y}z=0$ for all $x,y,z\in E$.

\begin{theorem}[Restriction and cross curvature]\label{thm:restriction}
The subspace $K_E$ is $G$-invariant, and $J_E$ is an ideal of $\g$.
If $B(E)\subset E^\perp$, then $B\in J_E$.  In particular,
\begin{equation}\label{eq:cross-in-J}
 P_{x,y}\in J_E\qquad \forall x\in E,\ y\in E^\perp,\ P\in\cM.
\end{equation}
Consequently,
\begin{equation}\label{eq:cross-kernel}
 G(B(\Rm))\subset K_E\qquad \forall B\in J_E.
\end{equation}
\end{theorem}
\begin{proof}
Let $Q\in K_E$ and $B\in\g$.  Evaluate~\eqref{eq:four-action} on
$x,y,z,t\in E$.  Each of its four terms is zero.  For example,
\[
 Q(Bx,y,z,t)=\ip{Q_{z,t}Bx}{y}
           =-\ip{Bx}{Q_{z,t}y}=0,
\]
and the other terms can be treated by pair symmetry and skew-adjointness.
Hence $r_E(B(Q))=0$, so $B(K_E)\subset K_E$.  Connectedness of $G$
gives $G$-invariance.
The induced representation on $\cM/K_E$ has kernel $J_E$ at the Lie
algebra level; therefore $J_E$ is an ideal.

If $B(E)\subset E^\perp$, every term on the right side of
\eqref{eq:four-action}, for an arbitrary $Q\in\cM$ and four arguments
in $E$, has one normal and three tangential arguments.  Geodesy, pair
symmetry and skew-adjointness make each term zero.  Thus $B\in J_E$.
For $B=P_{x,y}$ with $x\in E$, $y\in E^\perp$, pair symmetry gives
\[
 \ip{Bz}{t}=\ip{P_{z,t}x}{y}=0\qquad \forall z,t\in E,
\]
so $B(E)\subset E^\perp$.  Finally, $B(\Rm)\in K_E$ and $K_E$ is
$G$-invariant, proving~\eqref{eq:cross-kernel}.
\end{proof}

\begin{remark}\label{rem:quotient}
Although $E$ need not be $G$-invariant, its restriction kernel $K_E$ is.
Thus $\cM/K_E$ is a $G$-module, identified as a vector space with the
curvature restrictions visible on $E$.  The ideal $J_E$ is exactly the annihilator of the
quotient module.
\end{remark}

\subsection{Partial invariance is sufficient for trace vanishing}
The following lemma is the new ingredient in the orbit-span
induction.  Its proof resembles the Killing-form calculation in
Theorem~\ref{thm:kostant}, but with the error term that obstructs a Lie
algebra structure on $\g\oplus V$.

\begin{lemma}[Trace vanishing under an invariant ideal]\label{lem:trace}
Let $\g\subset\mathfrak{so}(V)$ be a Lie subalgebra, and let $Q$ be an
algebraic curvature tensor with $Q_{x,y}\in\g$.  Suppose
$\n\triangleleft\g$, is an ideal, and $B(Q)=0$ for every $B\in\n$.
If
\[
 A:=Q_{x,y}\in\n,\qquad Ax=Q_{x,y}x=0,
\]
then $A=0$.  Neither irreducibility nor invariance of $Q$ under all
of $\g$ is required.
\end{lemma}
\begin{proof}
On the vector space $\cL=\g\oplus V$ define endomorphisms
\begin{align}
 L_B: \cL \to \cL \text{ with } (C+z)&:=[B,C]+Bz&& \forall B,C\in\g,\ z\in V,\label{eq:LB}\\
 L_x: \cL\to \cL \text{ with } L_x(C+z)&:=Q_{x,z}-Cx&& \forall x,z\in V,\ C\in\g.\label{eq:Lx}
\end{align}
We do \emph{not} assume a Lie bracket on $\cL$.
Direct calculation using Bianchi gives
\begin{equation}\label{eq:commutator-error}
 [L_x,L_y]=L_{Q_{x,y}}+F_{x,y},\qquad
 F_{x,y}(C+z)=C(Q)_{x,y}.
\end{equation}
The result follows from the  calculation  on the two summands:
\begin{align*}
 [L_x,L_y]C
 &=-Q_{x,Cy}-Q_{Cx,y}
   =[Q_{x,y},C]+C(Q)_{x,y},\\
 [L_x,L_y]z
 &=-Q_{y,z}x+Q_{x,z}y=Q_{x,y}z.
\end{align*}

Since $A\in\n$ and $\n$ is an ideal,
$[A,C]\in\n$, and hence
\begin{equation}\label{eq:error-killed}
 F_{x,y}L_A(C+z)=[A,C](Q)_{x,y}=0.
\end{equation}
Since  $[L_A, L_x](C) =-C(Ax)$.  Using  $A(Q)=0$, namely 
$[A,Q_{x,z}]-Q_{x,Az}=Q_{Ax,z}$, we have that $[L_A, L_x](z)=Q_{Ax, z}$. Thus 
\begin{equation}\label{eq:LA-Lx}
 [L_A,L_x]=L_{Ax}=0.
\end{equation}

Using $A=Q_{x,y}$, equations~\eqref{eq:commutator-error}--\eqref{eq:LA-Lx}
and cyclicity of trace yield
\begin{align}
 \tr_{\cL}(L_A^2)
 &=\tr_{\cL}\bigl(([L_x,L_y]-F_{x,y})L_A\bigr)\notag\\
 &=\tr_{\cL}([L_x,L_y]L_A)
  =\tr_{\cL}([L_A,L_x]L_y)=0.\label{eq:trace-zero}
\end{align}
On the other hand, $L_A$ is block diagonal, so
\begin{equation}\label{eq:negative-trace}
 \tr_{\cL}(L_A^2)
 =\tr_{\g}(\ad_A^2)+\tr_V(A^2)
 =-\norm{\ad_A}_{\mathrm{HS}}^2-\norm{A}_{\mathrm{HS}}^2.
\end{equation}
Here $\|\cdot\|_{HS}$ means the Hilbert-Schmidt norm on the space of endomorphisms of $\g$ and $V$. 
  Equations~\eqref{eq:trace-zero} and
\eqref{eq:negative-trace} imply $A=0$.
\end{proof}

\begin{remark}\label{rem:trace-role}
The ideal assumption is used in~\eqref{eq:error-killed}, not merely the
single identity $A(Q)=0$.  The result, whose proof is a generalization of Kostant's calculation for
$(B_\cJ)|_{\mathfrak{g}} =B$,  is the substitute for the
implication ``zero Jacobi curvature implies zero curvature operator'' at
the stage where only an ideal of infinitesimal symmetries is known.
\end{remark}

\section{A single induction on the curvature-orbit span}\label{sec:induction}

We now prove Theorem~\ref{thm:main} by strong induction on $\kappa(S)$ from
\eqref{eq:index}.  The case $\kappa=0$ has $\Rm=0$.  If $\kappa=1$, the connected group
$G$ acts orthogonally on the one-dimensional space $G(\Rm)$ and therefore
acts trivially.  In both cases the rank conclusion follows from
Proposition~\ref{prop:flat}.

Fix $\kappa\geq2$ and assume Theorem~\ref{thm:main} for every irreducible system
whose curvature-orbit span has dimension less than $\kappa$, in every ambient
dimension.  Let $S=(V,\Rm,G)$ satisfy the hypotheses with $\dim\cM=\kappa$.
Set
\begin{equation}\label{eq:annihilator}
 \n:=\Ann_{\g}(\cM)
 =\{B\in\g:B(Q)=0\text{ for every }Q\in\cM\}.
\end{equation}
This is an ideal and is $\Ad_G$-invariant, since it is the kernel of the
infinitesimal representation on the $G$-module $\cM$.

\subsection{Induction assumption implies  restriction kernels are a fixed tensors}
\begin{lemma}\label{lem:inductive-kernel}
Let $E\subsetneq V$ be geodesic and suppose $r_E\neq0$.  Under the
induction hypothesis,
\begin{equation}\label{eq:kernel-fixed}
 K_E\subset\cM^G:=\{Q\, |\, g(Q)=Q \forall g\in G\},\qquad J_E=\n.
\end{equation}
In particular,
\begin{equation}\label{eq:cross-annihilator}
 \{P_{E,E^\perp}\, |\,\forall P\in\cM\}\subset\n.
\end{equation}
\end{lemma}
\begin{proof}
Theorem~\ref{thm:restriction} makes $K_E$ a $G$-invariant subspace.  Since
$r_E\neq0$, it is proper and $\dim K_E<m$.  For every $T\in K_E$,
\[
 G(T)\subset K_E,\qquad \g^T\subset\g^\Rm,\qquad G^T\subset G^\Rm.
\]
Both $\mathfrak g^T$ and $\mathfrak g^\Rm$ are ideals of the original Lie algebra
 $\mathfrak g$. Since $G$ is connected, compact, and irreducible on $V$, Lemma~\ref{lem:closed} implies 
 that their connected analytic subgroups $G^T$ and $G^\Rm$ are closed in $G$. 
 Moreover, $\mathfrak g^T\subset\mathfrak g^\Rm$ implies $G^T\subset G^\Rm$.
Thus $(V,T,G)$ has
smaller orbit-span dimension, remains irreducible, and has a nontransitive
curvature-generated group.  Induction gives $g(T)=T$ for all $g\in G$.
Therefore $K_E\subset\cM^G$.

If $B\in J_E$ and $T\in\cM$, then $B(T)\in K_E\subset\cM^G$, so
$B^2(T)=0$.  Equation~\eqref{eq:skew-norm} implies $B(T)=0$.  Hence
$J_E\subset\n$, and the reverse inclusion is immediate from the
definition.  Now~\eqref{eq:cross-annihilator} follows from
\eqref{eq:cross-in-J}.
\end{proof}

\subsection{Root centralizers span the ambient space}
Choose $a,b\in V$ with $\Rm(a,b,a,b)\neq0$, using
Lemma~\ref{lem:sectional}.  By Proposition~\ref{prop:flat}, choose a maximal
flat $W$ containing $a$.  Since $T_{\Rm}^{a,a}\neq0$, at least one of the forms
$\Lambda_k^{\Rm}$ on $W$ is nonzero.

Let $\mathscr E$ be the family of all distinct subspaces $E_{k,Q}$ attached
to this one $W$, as $Q$ ranges over $\cM$ and $\Lambda_k^Q\neq0$.
The family is nonempty.  Every $E\in\mathscr E$ is proper and geodesic,
contains $W$, and has $r_E\neq0$, by Proposition~\ref{prop:centralizers}.
Therefore Lemma~\ref{lem:inductive-kernel} gives
\begin{equation}\label{eq:all-root-cross}
 P_{E,E^\perp}\subset\n
 \qquad\forall E\in\mathscr E,\ P\in\cM.
\end{equation}

\begin{proposition}[Spanning by root centralizers]\label{prop:span}
Under the induction hypothesis,
\begin{equation}\label{eq:E-span}
 V=\sum_{E\in\mathscr E}E.
\end{equation}
Consequently at least two members of $\mathscr E$ are distinct, and
\begin{equation}\label{eq:E-intersection}
 \bigcap_{E\in\mathscr E}E=W,\qquad
 W^\perp=\sum_{E\in\mathscr E}E^\perp.
\end{equation}
\end{proposition}
\begin{proof}
Put $F=\sum_{E\in\mathscr E}E$ and take $z\in F^\perp$.
For every fixed $Q\in\cM$, each common eigenvector $X_k^Q$ with
$\Lambda_k^Q\neq0$ lies in its corresponding $E_{k,Q}\subset F$.
Thus $z$ is perpendicular to all such eigenvectors, and the spectral
decomposition gives
\begin{equation}\label{eq:zero-Jacobi}
 Q_{w,z}w=T_Q^{w,w}z=0
 \qquad(w\in W,\ Q\in\cM).
\end{equation}

Fix $w\in W$, $Q\in\cM$, and let $A=Q_{w,z}$.
Choose any $E\in\mathscr E$.  Since $w\in E$ and $z\in E^\perp$,
\eqref{eq:all-root-cross} gives $A\in\n$.  Moreover $\n$ annihilates
$Q$, and $Aw=0$ by~\eqref{eq:zero-Jacobi}.
Lemma~\ref{lem:trace}, applied with $x=w$ and $y=z$, yields $A=0$.
Consequently $z\in C(W)=W$.  Since $W\subset F$ and $z\perp F$, it
follows that $z=0$.  This proves~\eqref{eq:E-span}.

If all members of $\mathscr E$ were the same proper subspace, they could
not span $V$.  There are therefore two distinct members, whose intersection
is $W$ by Lemma~\ref{lem:duality}.  Every member contains $W$, so their
full intersection is $W$.  Taking orthogonal complements gives the second
identity in~\eqref{eq:E-intersection}.  In finite dimension the sums are
algebraic sums; each vector admits a finite decomposition.
\end{proof}

\subsection{Completion of the proof}
Let $P\in\cM$, $w\in W$, and $z\in V$.  By
\eqref{eq:E-intersection} write
\[
 z=z_W+\sum_{\alpha=1}^N z_\alpha,
 \qquad z_W\in W,\quad z_\alpha\in E_\alpha^\perp,
 \quad E_\alpha\in\mathscr E.
\]
Because $W$ is flat and $w\in E_\alpha$,
\[
 P_{w,z}=\sum_{\alpha=1}^N P_{w,z_\alpha}\in\n
\]
by~\eqref{eq:all-root-cross}.  Hence
\begin{equation}\label{eq:W-annihilator}
 P_{W,V}\subset\n\qquad \forall P\in\cM.
\end{equation}
For $g\in G$, equation~\eqref{eq:group-action} gives
\[
 P_{gw,z}=\Ad_g\bigl((g^{-1}(P))_{w,g^{-1}z}\bigr)\in\n,
\]
since $g^{-1}(P)\in\cM$ and $\n$ is $\Ad_G$-invariant.
Irreducibility and $W\neq0$ imply
\[
 \Span\{gW:g\in G\}=V.
\]
By linearity, $P_{x,y}\in\n$ for all $P\in\cM$ and $x,y\in V$.
Therefore
\[
 \g^\Rm\subset\n.
\]
Lemma~\ref{lem:ideal} also gives $\g_\Rm\subset\n$.  Thus
$\n=\g$, so $B(\Rm)=0$ for every $B\in\g$.  Connectedness of $G$
implies $g(\Rm)=\Rm$ for all $g\in G$.
The flat $W$ has dimension at least two, proving the rank assertion and
completing the induction for Theorem~\ref{thm:main}.

\begin{remark}[The one induction via $\dim(G(\Rm))$]\label{rem:single-induction}
Lemma~\ref{lem:inductive-kernel} is the  only place where the induction assumption is used, through
\[
 G(T)\subset K_E\subsetneq G(\Rm).
\]
The trace lemma, the flat lemma, and the construction and geodesy of root
centralizers are independent of the induction. 
\end{remark}

\subsection*{Acknowledgments}
The author thanks N.~Wallach for the analytic argument for the existence of
flats, and H.~Wu, Y.~Niu, and F.~Zheng for their interests, I.~Agricola  for
helpful suggestions.

\end{document}